\documentclass{biometrika}

\usepackage[plain,noend]{algorithm2e}
\usepackage{natbib,color,bm,amsmath,amssymb,graphicx,float,booktabs,subfig,multirow,threeparttable,textcomp,times}
\makeatletter
\renewcommand{\algocf@captiontext}[2]{#1\algocf@typo. \AlCapFnt{}#2} % text of caption
\def\@algocf@capt@plain{top}
\renewcommand{\algocf@makecaption}[2]{%
  \addtolength{\hsize}{\algomargin}%
  \sbox\@tempboxa{\algocf@captiontext{#1}{#2}}%
  \ifdim\wd\@tempboxa >\hsize%     % if caption is longer than a line
    \hskip .5\algomargin%
    \parbox[t]{\hsize}{\algocf@captiontext{#1}{#2}}% then caption is not centered
  \else%
    \global\@minipagefalse%
    \hbox to\hsize{\box\@tempboxa}% else caption is centered
  \fi%
  \addtolength{\hsize}{-\algomargin}%
}
\makeatother

\def\Bka{{\it Biometrika}}

\begin{document}

\newcommand{\red}{\color{red}}
\newcommand{\blue}{\color{blue}}
\newcommand{\mage}{\color{magenta}}

\def\logit{\text{\rm logit }}
\def\expit{\text{\rm expit }}
\newcommand\independent{\protect\mathpalette{\protect\independenT}{\perp}}
\def\independenT#1#2{\mathrel{\rlap{$#1#2$}\mkern2mu{#1#2}}}
\makeatletter
\newcommand*{\indep}{%
\mathbin{%
\mathpalette{\@indep}{}%
}%
}
\newcommand*{\nindep}{%
\mathbin{% % The final symbol is a binary math operator
\mathpalette{\@indep}{\not}% \mathpalette helps for the adaptation
% of the symbol to the different math styles.
}%
}
\newcommand*{\@indep}[2]{%
% #1: math style
% #2: empty or \not
\sbox0{$#1\perp\m@th$}% box 0 contains \perp symbol
\sbox2{$#1=$}% box 2 for the height of =
\sbox4{$#1\vcenter{}$}% box 4 for the height of the math axis
\rlap{\copy0}% first \perp
\dimen@=\dimexpr\ht2-\ht4-.2pt\relax
% The equals symbol is centered around the math axis.
% The following equations are used to calculate the
% right shift of the second \perp:
% [1] ht(equals) - ht(math_axis) = line_width + 0.5 gap
% [2] right_shift(second_perp) = line_width + gap
% The line width is approximated by the default line width of 0.4pt
\kern\dimen@
{#2}%
% {\not} in case of \nindep;
% the braces convert the relational symbol \not to an ordinary
% math object without additional horizontal spacing.
\kern\dimen@
\copy0 % second \perp
} 
\makeatother

\jname{Biometrika}
%%% The year, volume, and number are determined on publication
\jyear{2015}
\jvol{}
\jnum{}
%% The \doi{...} and \accessdate commands are used by the production team
%\doi{10.1093/biomet/asm023}
%\accessdate{Advance Access publication on 31 July 2012}
\copyrightinfo{\Copyright\ 2012 Biometrika Trust\goodbreak {\em Printed in Great Britain}}
%
%%% These dates are usually set by the production team
%\received{April 2012}
%\revised{September 2012}

%% The left and right page headers are defined here:
\markboth{Wang Miao \and Eric Tchetgen Tchetgen}{Miscellanea}
%
%%% Here are the title, author names and addresses
%\title{Some notes on Biometrika style}
%
%\author{A. C. DAVISON, R. GESSNER}
%\affil{Institute of Mathematics, Ecole Polytechnique F\'ed\'erale de Lausanne, Station 8,\\ 1015 Lausanne, Switzerland \email{editor.biometrika@epfl.ch} \email{biometrika@epfl.ch}}
%
%\author{\and D. M. TITTERINGTON}
%\affil{Department of Statistics, University of Glasgow, Glasgow G12 8QQ, U.K. \email{mike@stats.gla.ac.uk}}
%
%\maketitle

\title{On Varieties of Doubly Robust Estimators Under Missingness Not at Random With a Shadow Variable}

\author{Wang Miao}
\affil{Beijing International Center for Mathematical Research, Peking University,\\ Beijing 100871, P.R.C. \email{mwfy@pku.edu.cn}}
\date{\today}

\author{\and Eric Tchetgen Tchetgen}
\affil{Department of Biostatistics, Harvard University, Boston,  Massachusetts 02115, U.S.A. \email{etchetge@hsph.harvard.edu}}

\maketitle

\begin{abstract}
Suppose we are interested in the mean of an outcome variable missing not at random. Suppose however that one has available a fully observed shadow variable, which is associated with the outcome but independent of the missingness process conditional on covariates and the possibly unobserved outcome. Such a  variable may be a proxy or a mismeasured version of the outcome available for all individuals.
We have previously established necessary and sufficient conditions for identification of the full data law in such a setting, and have described  semiparametric estimators including a doubly robust estimator of the outcome mean. Here, we propose  two alternative   doubly robust estimators  for the outcome mean, which may be viewed as extensions of analogous methods under missingness at random, but  enjoy different properties.  We  assess correctness of the required working models via straightforward goodness-of-fit tests.

\end{abstract}


\section{Introduction}
Doubly robust methods are  designed to mitigate estimation bias due to model misspecification in  observational studies and imperfect experiments. Such methods have grown in popularity in recent years for estimation with missing data and other forms of coarsening \citep{robins1994estimation,scharfstein1999adjusting,van2003unified,bang2005doubly,tsiatis2007semiparametric}. 
There  exist  various constructions  of doubly robust estimators for the mean of  an outcome that is missing at random; see  \cite{kang2007demystifying}.
In contrast, for data missing not at random,    difficulty of identification   undermines one's ability to obtain accurate inferences, and doubly robust estimation is  far more challenging.
Identification of a full data model means that, the parameters indexing the model are uniquely determined by the observed data,  i.e., the data that are actually observed on the individuals. 
Statistical inference based on non-identifiable models may be misleading and of limited interest in practice; see \cite{miao2014normal}. 
Under missingness at random,  the full data law, i.e., the joint distribution of  all variables of interest, is nonparametrically identified from the observed data. However, under missingness not at random, identification  is no longer  possible  without further restrictions on the missingness process.  
Although no general identification results are available for data missing not at random,   one may  identify the full data law  under specific  assumptions.
Building on earlier work by  \cite{d2010new}, \cite{wang2014instrumental} and \cite{zhao2014semiparametric}, \cite{2015arXiv150902556M}   used a  fully observed shadow variable  to establish a general identification framework for data missing not at random.
Such a variable is   associated with the outcome conditional on covariates, but independent of the missingness  conditional on   covariates and the outcome \citep{kott2014}; it may be available in many empirical studies, where a fully observed proxy or a mismeasured version of the outcome is available.
For example, in a study of mental health of children in Connecticut \citep{zahner1992children,ibrahim2001using}, researchers were interested in evaluating the prevalence of students with abnormal psychopathological status based on their teacher's assessment, which was subject to missingness.  A separate parent report available for all children in the study, is a proxy for   the teacher's assessment, but is unlikely to be related to the teacher's response rate conditional on covariates and her assessment of the student; in this case the parental assessment constitutes a valid shadow variable.  Other examples can be found in \cite{2015arXiv150902556M} and \cite{wang2014instrumental}. 

Throughout, we let $Y$ denote the outcome, $R$ is its missingness indicator with $R=1$ if $Y$ is observed, otherwise $R=0$, and let $X$ denote fully observed covariates.  Suppose that one has also fully observed a shadow variable $Z$ that satisfies 
\begin{assumption}\label{assump:anci}
(i)$Z$ $\nindep$ $Y$ $\mid$ $X$; (ii)  $Z$ $\indep$ $R$ $\mid$ $(Y,X)$.
\end{assumption}
Assumption \ref{assump:anci} formalizes  the idea that,  the shadow  variable only affects the missingness through its association with the outcome.
We provide  a directed acyclic graph  in the Supplementary Material that can help to understand the assumption.
The shadow variable introduces additional conditional independence conditions, which impose further restrictions on the missingness  process, and thus  
provides better opportunity for identification despite the fact that data may be missing not at random.
\cite{2015arXiv150902556M} presented a brief review of such problems, and gave necessary and sufficient conditions as well as    sufficient conditions for identification with a shadow variable.
In particular,  if the outcome is binary, the full data law  is  identifiable with  a binary shadow variable. But for a continuous outcome, a binary shadow variable does not impose enough restrictions to identify the full data law; see the Supplement Material for a counterexample. Identification  for a continuous outcome  requires at least one continuous shadow variable,  but even then, additional conditions are needed. 
We consider a  location-scale model for the density function: 
\begin{eqnarray}\label{mdl:1}
f(y\mid x,z,r) = \frac{1}{\sigma_r(z,x)}f_{r}\left\{\frac{y-\mu_r(z,x)}{\sigma_r(z,x)}\right\},\quad  r=0,1,
\end{eqnarray}
with  unrestricted functions $\mu_r$ and $\sigma_r$, and   density functions $f_{r}$.   Under certain regularity conditions summarized in the Appendix,
we have previously proved identification of the full data law if either $f(y\mid x,z,r=1)$ or $f(y\mid x,z,r=0)$ follows model \eqref{mdl:1},  even if the  missingness process is unrestricted \citep{2015arXiv150902556M}.
Aside for  Assumption 1,  model (1) includes  many commonly-used models,  for instance,  Gaussian models, and thus  essentially demonstrates that lack of identification is not an issue in many familiar situations.  
However,  one cannot understate the central role of the shadow variable for identification. Without such a variable, identification is no longer guaranteed for model \eqref{mdl:1},  even if one were to assume a parametric missingness model. For additional and extensive discussion about identification under missingness not at random with a shadow variable, see  \cite{2015arXiv150902556M} and \cite{wang2014instrumental}.

With models satisfying the corresponding identification conditions, previous authors have developed several non-doubly robust estimators.
Among them, inverse probability weighted estimation \citep{wang2014instrumental} and pseudo-likelihood  estimation  \citep{zhao2014semiparametric}  are sensitive to model misspecification; and  nonparametric estimation  \citep{d2010new} requires  an unrealistic  large sample size  for reasonable performance when  the covariate dimension is moderate to large.
In contrast, a doubly robust approach remains consistent and asymptotically normal  under partial misspecification.
Specifically,  \cite{2015arXiv150902556M} developed a doubly robust estimator  based on a three-part model for the full data: a model for the joint distribution of  the outcome and  the shadow variable in complete cases; a model for  the propensity score  evaluated at a reference value of the outcome;  and  a log odds ratio model  encoding the association of the outcome and the missingness process.  
Under correct specification of the log odds ratio model, the doubly robust estimator is consistent if either  of the other two models is correct, but not necessarily both.
However, the construction of a doubly robust estimator  is not unique.
In this paper, we develop  two alternative  doubly robust estimators of the outcome mean that  enjoy different properties, and we  compare them both in theory and via simulations reported in the Supplementary Material.

\section{Doubly robust estimators}
Under Assumption \ref{assump:anci}, we  factorize the conditional density function of $(Z,Y,R)$ given $X$ as  
\begin{eqnarray}\label{dis}
f(z,y,r\mid x)  =  c(x)\exp\{(1-r){\rm OR}(y\mid x)\}{\rm pr}(r\mid y=0,x)f(z,y\mid r=1,x), 
\end{eqnarray}
where $c(x)={\rm pr}(r=1\mid x)/{\rm pr}(r=1\mid y=0,x)$; ${\rm pr}(r=1\mid y=0,x)$ is  the response probability  evaluated at the reference level $y=0$, and is referred to as the   baseline propensity score;  $f(z,y\mid r=1,x)$ is the joint density function of $(Z,Y)$ conditional on $X$ among the complete cases, i.e., the subset with $r=1$, and is referred to as the baseline outcome density;  
\begin{eqnarray*}
{\rm OR}(y\mid x) & = & \log \frac{{\rm pr}(r=0\mid y,x) {\rm pr}(r=1 \mid y=0,x)}{{\rm pr} (r=0 \mid y=0,x) {\rm pr}(r=1 \mid y,x)},
\end{eqnarray*}
is the log of the conditional odds ratio function relating $Y$ and $R$ given $X$ with $E[\exp\{{\rm OR}(y \mid x)\} \mid r=1,x] < \infty$ and   ${\rm OR}(y=0 \mid x)=0$.
For a continuous outcome, we require that  $f(z,y \mid r=1,x)$ satisfies model \eqref{mdl:1} to guarantee  identification.
For estimation, we specify separate parametric   models  ${\rm pr}(r=1 \mid y=0,x;\alpha)$,  $f(z,y \mid r=1,x;\beta)$, and ${\rm OR}(y \mid x;\gamma)$.
We suppose throughout that ${\rm OR}(y\mid x;\gamma)$ is correctly specified, which can be achieved  by specifying a relatively flexible model, or following the approach suggested by  \cite{higgins2008imputation} if  information on the  reasons  for missingness are available.
From \eqref{dis}, we have the following identities:
\begin{align}
{\rm pr}(r=1\mid y,x)  &=&&   \frac{{\rm pr}(r=1 \mid y=0,x)}{{\rm pr}(r=1 \mid y=0,x)+\exp\{{\rm OR}(y \mid x)\}\{{\rm pr}(r=0 \mid y=0,x)\}},\label{eq:propen}\\
f(z,y \mid r=0,x)  &=&&   \frac{\exp\{{\rm OR}(y \mid x)\}}{E[\exp\{{\rm OR}(y \mid x)\} \mid r=1,x]}f(z,y \mid r=1,x),\label{eq:anci}\\
E(y \mid r=0,x) &= &&  \frac{ E[\exp\{{\rm OR}(y \mid x)\}y \mid r=1,x]}{ E[\exp\{{\rm OR}(y \mid x)\} \mid r=1,x]}. \label{eq:propen0}
\end{align}
The propensity  score, and its reciprocal, i.e., the inverse probability weight function $W(x,y;\alpha,\gamma)= 1/{\rm pr}(r=1\mid x,y;\alpha,\gamma)$, are determined by the baseline propensity score model ${\rm pr}(r=1 \mid x,y=0;\alpha)$   and the log odds ratio model  ${\rm OR}(y \mid x;\gamma)$ as in  \eqref{eq:propen}; the conditional outcome mean  among the incomplete cases  $ E(y \mid r=0,x;\beta,\gamma)$ is  determined by the baseline outcome model and the log odds ratio model as in  \eqref{eq:propen0}.

Estimation of $\beta$ only involves the complete cases. Let  $\widehat E$ denote  the empirical mean, we solve
\begin{eqnarray}
&& \widehat E  \{rS(z,y,x;\widehat \beta)\}=0, \label{eq:nui1}
\end{eqnarray}
with  score function $S(z,y,x;\beta)=\partial \log\{P(z,y \mid r=1,x;\beta)\}/\partial \beta$. 
Estimation of $\widehat\alpha$ and $\widehat\gamma$ is motivated from  a classic estimating equation following the fact that the respective weighted mean of any vector functions $G(x,y)$ and $H(x)$  among the  complete cases  equals their population mean:
$\widehat E  [ \{W(x,y;\widehat\alpha,\widehat\gamma ) r -1\} 		\{G(x,y)^T,H(x)^T\}^T ] =0$, where    $G(x,y)$ and $H(x)$ are  user-specified vector functions of dimension equal to that of  $\gamma$ and $\alpha$, respectively, and satisfy  $E[\partial W(x,y;\alpha,\gamma ) r /\partial (\alpha,\gamma)\{G(x,z)^T,H(x)^T\}]$ is nonsingular for all $(\alpha,\gamma)$. For example, if  ${\rm pr}(r=1|y,x;\alpha,\gamma)$ follows a logistic model and thus $W(x,y;\alpha,\gamma)=1+\exp\{-(1,x^T)\alpha-\gamma y\}$, we may naturally choose $G(x,y)=y$ and $H(x)=(1,x^T)^T$.  Because  $y$ is missing for $r=0$, the classic estimating equation is not feasible. However,   Assumption \ref{assump:anci} allows us to replace $y$  with  the shadow variable $z$ and  to replace
 $G(x,y)$ with $G(x,z)$.  
To further derive doubly robust estimators, we incorporate the baseline outcome model into the  estimating equation for $(\alpha,\gamma)$. Let $G_1(x,z;\beta,\gamma)=G(x,z)-E\{G(x,z)|r=0,x;\beta,\gamma\}$, we solve
\begin{eqnarray}
&&\widehat E  [ \{W(x,y;\widehat\alpha ,\widehat\gamma ) r -1\} \{G_1(x,z;\widehat\beta,\widehat\gamma)^T,H(x)^T\}^T ] =0,\label{eq:nui2}
\end{eqnarray} 
with $G(x,z)$ and $H(x)$  such that $E[\partial W(x,y;\alpha,\gamma ) r /\partial (\alpha,\gamma)\{G_1(x,z;\beta,\gamma)^T,H(x)^T\}]$ is nonsingular for all $(\alpha,\beta,\gamma)$.
The shadow variable $Z$ is used as a proxy of $Y$, thus, a choice of $Z$ that is highly correlated with $Y$ is desirable for  the purpose of efficiency maximization.

Using  $(\widehat\alpha,\widehat\beta,\widehat\gamma)$ obtained from equations \eqref{eq:nui1} and \eqref{eq:nui2}, we construct three different estimators for the outcome mean that are consistent if   either  the baseline outcome model
or the baseline propensity score model is correctly specified, together with the log odds ratio model.

A regression estimator with residual bias correction  was previously described by \cite{2015arXiv150902556M}. We use the weighted residual to correct the bias of the conditional mean among incomplete cases.
Let $M_0(x;\widehat\beta,\widehat\gamma)= E(y \mid r=0,x;\widehat \beta,\widehat\gamma )$,  the estimator is
\begin{equation*}
\widehat\mu_1 = \widehat E [ W (x,y;\widehat\alpha,\widehat\gamma) r
\{ y-M_0(x;\widehat \beta,\widehat\gamma) \}  
+M_0(x;\widehat \beta,\widehat\gamma) ].
\end{equation*}

A Horvitz--Thompson estimator with extended  weights  employs   an extended baseline propensity score model and an extended weight function.  
The extended baseline propensity score model with unknown parameter $\phi$  satisfies  ${\rm pr}_{\rm ext}(r=1 \mid y=0,x;\phi)={\rm pr}(r=1 \mid y=0,x;\widehat\alpha)$  only at   $\phi=0$. For example, we can specify 
\[{\rm pr}_{\rm ext}(r=1 \mid y=0,x;\phi)=\frac{{\rm pr}(r=1 \mid y=0,x;\widehat\alpha)}{{\rm pr}(r=1 \mid y=0,x;\widehat\alpha)+\exp\{\phi g(x)\}{\rm pr}(r=0 \mid y=0,x;\widehat\alpha)},\]
with user-specified scalar function $g(x)$.
The extended weight function   $W_{\rm ext}(x,y;\phi)$, and its reciprocal is determined    as in  \eqref{eq:propen}  with $OR(y|x)$  and ${\rm pr}(r=1 \mid y=0,x)$ replaced by $OR(y\mid x;\widehat\gamma)$ and  ${\rm pr}_{\rm ext}(r=1 \mid y=0,x;\phi)$ respectively. We estimate  $\phi$ by solving 
\begin{eqnarray}
\widehat E [\{W_{\rm ext}( x,y;\widehat\phi) r -1\} \{ 
	M_0(x;\widehat\beta,\widehat\gamma) - \widehat\mu_{\rm reg}\} ] =0, \label{eq:ht}
\end{eqnarray}
with   previously obtained $(\widehat\beta,\widehat\gamma)$ and $\widehat\mu_{\rm reg} = \widehat E  \{(1-r)M_0(x;\widehat\beta,\widehat\gamma)+ry\}$.
The Horvitz--Thompson estimator with extended  weights   is 
\begin{eqnarray*}
\widehat \mu_{2} = \widehat E  \left\{\frac{W_{\rm ext}(x,y;\widehat\phi )r}{\widehat E  \{W_{\rm ext}( x,y;\widehat\phi )r\}} y\right\}. 
\end{eqnarray*}

A regression estimator with an extended outcome model involves an extended outcome model $M_{\rm 0ext}(x;\psi)$ with parameter $\psi$ satisfying  $M_{\rm 0ext}(x;\psi)=M_{0}(x;\widehat\beta,\widehat\gamma)$   only at $\psi=0$. If $M_0(x;\widehat\beta,\widehat\gamma)=\lambda\{Q(x;\widehat\beta,\widehat\gamma)\}$ for some inverse  link  $\lambda$ and some function $Q$, we can specify $M_{\rm 0ext}(x;\psi)=\lambda\{Q(x;\widehat\beta,\widehat\gamma)+\psi q(x)\}$ with a scalar function $q(x)$. 
We estimate $\psi$   by solving 
\begin{eqnarray}
\widehat E  [ \{W( x,y;\widehat\alpha,\widehat\gamma)  -1 \} r\{y-M_{\rm 0ext}(x;\widehat\psi) \}]=0,\label{eq:rext}
%&&\widehat E  [r\{y-M_1(x;\widehat\beta) - \widehat\psi_1 g(X)  \}]=0,\label{eq:regipw3}
\end{eqnarray}
with previously obtained $(\widehat\alpha,\widehat\gamma)$. The regression estimator  with an extended outcome model is 
\begin{eqnarray*}
\widehat \mu_3= \widehat E     \{(1-r)M_{\rm 0ext}(x;\widehat\psi)  + ry\}.
\end{eqnarray*}
%In special situations, the three estimators $\widehat\mu_1,\widehat\mu_2,\widehat\mu_3$ can be equivalent. For example, if the components of $H(x,z)$ include both a constant function and  $M_0(x;\widehat\beta,\widehat\gamma)$,  then $\widehat\mu_1$ equals $\widehat\mu_2$; 
%if  $M_{\rm 0ext}(x;\phi)=M_0(x;\widehat\beta,\widehat\gamma)+\phi$, then $\widehat\mu_1$ equals $\widehat\mu_3$.
%However, 
The  estimators $\widehat\mu_1,\widehat\mu_2$ and $\widehat\mu_3$  may have very different characteristics, although, all three estimators are  doubly robust.
\begin{theorem}\label{thm:dr1}
Under Assumption 1, if the log odds ratio model ${\rm OR}(y \mid x;\gamma)$ is correct, and  the probability limit of   equations  \eqref{eq:nui1}, \eqref{eq:nui2},  \eqref{eq:ht} and \eqref{eq:rext} has a unique solution, then
$\widehat \mu_1$, $\widehat \mu_2$ and  $\widehat \mu_3$ are consistent if either  $f(z,y \mid r=1,x;\beta)$ or ${\rm pr}(r=1 \mid y=0,x;\alpha)$ is correctly specified.
\end{theorem}

The  extended models not only provide  double robustness, but also provide  a  strategy  to check if the working models are correct.
We prove in the Appendix that if the baseline propensity score model is correct, $\widehat\phi$ converges to $0$ in probability; and if the baseline outcome model is correct, $\widehat\psi$ converges to $0$ in probability.
Therefore, one may use this property to assess whether the working models are correctly specified by checking  whether $\widehat\phi$ and $\widehat\psi$ are within sampling variability of zero, respectively.  
However, one should acknowledge that  the space of possible departures from the assumed model may be prohibitively large relative to the proposed test
so that the resulting goodness-of-fit test will generally have good power against certain alternatives but not in all possible directions away from the specified working model. We explore the power of the proposed goodness-of-fit test via a simulation study  in the Supplementary Material.

All three doubly robust estimators rely on a correct log odds ratio model, since inference about the law of $Y$  requires an accurate evaluation  of the dependence  between the missingness process and the outcome, which is captured by the log odds ratio model ${\rm OR}(y \mid x;\gamma)$.
%{\red Previous authors have used such models for imputation\citep{higgins2008imputation}, or doubly robust estimation\citep{vansteelandt2007estimation,robins2008higher} of  data missing not at random.
%}
To the best of our knowledge, with the exception of \cite{2015arXiv150902556M},  previous doubly robust estimators  have assumed that this log odds ratio is known, either to equal the null value of $0$ under missingness at random \citep{bang2005doubly,tsiatis2007semiparametric,van2003unified}, or to be of a known functional form with  no unknown parameters \citep{vansteelandt2007estimation,robins2008higher}. We have relaxed these  more stringent assumptions.

\section{Relation to previous doubly robust estimators and  comparisons}
Previous doubly robust estimators under missingness at random can be viewed as special cases of our estimators.
Under missingness at random,   ${\rm OR}(y \mid x)=0$, ${\rm pr}(r=1 \mid x,y=0)={\rm pr}(r=1 \mid x)$, the inverse probability weight function $W(x;\alpha)=1/{\rm pr}(r=1 \mid x;\alpha)$ does not vary with $y$,  and the conditional mean among the population  $M(x;\beta)$ equals  that among the incomplete cases $M_0(x;\beta,\gamma)$.
The estimator  $\widehat\mu_1'=\widehat E  [W(x;\widehat\alpha)r\{y-M(x;\widehat\beta)\}+ M(x;\widehat\beta)]$ of \cite{kang2007demystifying} is a special case of the regression estimator with  residual bias correction;
the estimator $\widehat \mu_2' = \widehat E  [W_{\rm ext}(x;\widehat\phi )r/\widehat E  \{W_{\rm ext}( x;\widehat\phi )r\} y]$  proposed by \cite{robins2007comment},  with  an extended logistic propensity score model  $\logit {\rm pr}_{\rm ext}(r=1 \mid x;\phi)=(1,x^T)\widehat\alpha+\phi g(x)$,
is a special case of  the Horvitz--Thompson estimator with extended  weights;
the estimator $\widehat\mu_3' =\widehat E \{M_{\rm ext}(x;\widehat\psi)\}$  proposed by \cite{robins2007comment}, with an extended outcome model $M_{\rm ext}(x;\widehat\psi)$ satisfying  $\widehat E [W( x;\widehat\alpha)r\{y-M_{\rm ext}(x;\widehat\psi)\}]=0$ and
$\widehat E [r\{y-M_{\rm ext}(x;\widehat\psi)\}]=0$, is a special case of the regression estimator with an extended outcome model.

The three proposed doubly robust estimators  enjoy some of the properties of their  missingness at random analogs.
The  estimator $\widehat\mu_2$  is  a convex combination of the observed outcome values.
It satisfies the  boundedness  property \citep{robins2007comment} that the estimator falls in the  parameter space for the outcome mean almost surely.
Such estimators are preferred when the inverse probability weights are highly variable, because they  rule out estimates outside the sample space. 
%The asymptotic bias of $\widehat \mu_{HT-RW}$  is
%\[Bias_{HT-EXT}=1/\mathbb E\{W^*(x,y)/W^0(x,y)\}\cdot \mathbb E[\{W^*(x,y)/W^0(x,y)-1\}\{M_0^0(x)-\mathbb E(Y)\}], \]
%which does not exceed $\max[ \mid M_0^0(x)-\mathbb E\{M_0^0(x)\} \mid ]+ \mid \mathbb E\{M_0^0(x)-\mathbb E(Y) \mid /{\rm pr}(r=1)$.
Boundedness is not guaranteed for  $\widehat\mu_1$.
If the range of $M_{\rm 0ext}(x;\psi)$  is contained in the sample space of the outcome, $\widehat\mu_3$ also satisfies the boundedness condition, but this does not hold in general. For example, if the outcome is continuous, and  $M_{\rm 0ext}(x;\psi)=M_0(x;\widehat\beta,\widehat\gamma)+\psi$, the range of $\widehat\mu_3$   may be outside the  sample space of the outcome mean.

%Our approach requires a shadow variable to identify the log odds ratio model. But the shadow variable is not necessary for data missing at random because the full data law is  nonparametrically identified.
%Aside from that, 
The three proposed estimators offer certain improvements  in term of  bias when both models are misspecified.
The asymptotic  bias of $\widehat\mu_1$ can be written as
\begin{equation*}
{\rm Bias}_1= E[\{W(x,y;\alpha^*,\gamma^*)r-1\}\{y-M_0(x;\beta^*,\gamma^*)\}],
\end{equation*}
and the asymptotic bias  of $\widehat\mu_3$  has the same form with  $M_0(x;\beta,\gamma^*)$ replaced by $M_{\rm 0ext}(x;\psi^*)$, with probability limits   $(\alpha^*,\beta^*,\gamma^*,\psi^*)$   of the corresponding estimators.
The bias  is driven by the degree of  misspecification of both the weight function and the conditional mean  among the incomplete cases.
As pointed out by \cite{robins2007comment} and \cite{vermeulen2014biased}, without further restrictions on the inverse probability weights, ${\rm Bias}_1$  gets inflated in  regions with  large weights.
However, if the components of $H(x)$ in equation \eqref{eq:nui2} include a constant function, then $E \{W(x,y;\alpha^*,\gamma^*)r\}=1$, which restricts the amount of  variability of the inverse probability weights. Thus, ${\rm Bias}_1$  does not explode with large weights.

In simulation studies, we found that the three  doubly robust estimators approximate the true outcome mean if either of  the baseline models is correct, but they are biased if neither baseline model is correct. 
For the case with moderately variable weights, the relative magnitude of the bias depends on the specific data generating process, but for the case with highly-variable weights,  the Horvitz--Thompson estimator with extended  weights has smaller bias.
If the baseline outcome model is correct, the parameter of the extended  outcome model, $\widehat\psi$  is close to $0$; and if  the baseline propensity score model is correct, the parameter of the extended  weight model, $\widehat\phi$  is close to $0$.
We also perform formal tests of the   null hypotheses  $\mathbb H_0:\ \phi=0$ and $\mathbb H_0:\ \psi=0$ respectively under level $0.05$. The results show an empirical type I error approximating  $0.05$  if the required  baseline propensity score model or   baseline outcome model  is correct, respectively (i.e., the true value of $\phi$ and $\psi$ equals $0$ respectively). Such tests   have good power   in moderate  samples if  the required  model is incorrect, respectively.
We recommend  the proposed  hypothesis tests to  check for severe  misspecification of the baseline  models  in practice.

\section{Discussion}
Extensions of the doubly robust methods described in this work   to other functionals, such as a  parameter $\delta$  solving a full data estimating equation $ E\{U(z,y,x;\delta)\}=0$, can be achieved by replacing $Y$ with $U$ wherever $Y$ occurs in the  estimating equations and solving the doubly robust estimating equation for the parameter of interest.
The  methods also have potential application in  related areas, such as  longitudinal data analysis and causal inference.

\section*{Acknowledgement}
The work is partially supported by the China Scholarship Council and the National Institute of Health.
The authors are grateful to the referees and the editor for their helpful comments.
\section*{Supplementary material}
Supplementary material available at \Bka\  online includes the proof of a lemma, a counterexample to identification with a continuous outcome, a graph model for the shadow variable, and  simulation studies.

\appendix

\section*{Appendix}
\subsection*{Proof of Theorem \ref{thm:dr1}}
We need the following lemma, which we prove in the Supplementary Material.
\begin{lemma}\label{lem:1}
Under Assumption 1, suppose that the log odds ratio model is correct, and that the probability limit of equations \eqref{eq:nui1} and \eqref{eq:nui2}  has a unique solution. For  any square integrable vector function $D(z,y,x)$,  scalar function $V(x)$,  and  $(\widehat\alpha,\widehat\beta,\widehat\gamma)$ solving equations \eqref{eq:nui1} and \eqref{eq:nui2},
\begin{enumerate}
\item[(i)] if  ${\rm pr}(r=1 \mid y=0,x;\alpha)$  is correct,  then $\widehat E  [\{W( x,y;\widehat\alpha,\widehat\gamma) r -1\}D(z,y,x)]$ converges to $0$ in probability;
\item[(ii)] if $f(z,y \mid r=1,x;\beta)$ is correct, then
$ \widehat E [r\exp\{{\rm OR}(y \mid x;\widehat\gamma)\} V(x) \{D(z,y,x)-E[D(z,y,x) \mid r=0,x;\widehat\beta,\widehat\gamma]\}]$ converges to $0$ in probability;
\item[(iii)] if either of the baseline models is correct, then
$\widehat E  [\{W(x,y;\widehat\alpha,\widehat\gamma)r-1\}\{D(z,y,x)-E [D(z,y,x) \mid r=0,x;\widehat\beta,\widehat\gamma]\}]$ converges to $0$ in probability.

\end{enumerate}
\end{lemma}

\begin{proof}[of Theorem~\ref{thm:dr1}]
Suppose that the log odds ratio model is correctly specified,  and  that the probability limit of the estimating equations has a unique solution.
\begin{enumerate}
\item Double robustness of $\widehat\mu_1$. If either of the baseline models is correct, from (iii) of Lemma \ref{lem:1}, $\widehat E  [\{W(x,y;\widehat\alpha,\widehat\gamma)r-1\}\{y- E (y \mid r=0,x;\widehat\beta,\widehat\gamma)\}]$ converges to $0$, therefore $\widehat E  [ W (x,y;\widehat\alpha,\widehat\gamma) r \{ y-M_0(x;\widehat \beta,\widehat\gamma) \}   +M_0(x;\widehat \beta,\widehat\gamma) ]$   converges  to the true outcome mean.

\item  Double robustness of $\widehat\mu_2$. From (i) of Lemma \ref{lem:1}, if the baseline propensity score model is correct,  $\widehat E  [\{W_{\rm ext}( x,y;\phi=0) r -1\}\{M_0(x;\widehat\beta,\widehat\gamma)-\widehat\mu_{\rm reg}\}] =  \widehat E  [\{W( x,y;\widehat\alpha, \widehat\gamma) r -1\}\{M_0(x;\widehat\beta,\widehat\gamma)-\widehat\mu_{\rm reg}\}]$ converges to $0$, i.e., $\phi=0$ is a solution of the probability limit of equation \eqref{eq:ht}. Thus, the solution of equation \eqref{eq:ht} $\widehat\phi$ converges to $0$,
and $\lim_{n\rightarrow +\infty} \widehat E \{W_{\rm ext}( x,y;\widehat\phi)r \}=1$, 
$\lim_{n\rightarrow +\infty} \widehat E \{W_{\rm ext}( x,y;\widehat\phi)r y\}=\lim_{n\rightarrow +\infty} \widehat E \{W( x,y;\widehat\alpha, \widehat\gamma)r y\}= E(Y)$.  
If the baseline outcome  model is correct, $\widehat E [(1-r)\{y-M_0(x;\widehat\beta,\widehat\gamma)\}]$ converges to 0; $\widehat\mu_{\rm reg} = \widehat E  [(1-r)M_0(x;\widehat\beta,\widehat\gamma)+ry]$  converges to the true outcome mean; and  $\widehat E (y-\widehat\mu_{\rm reg})$ converges to $0$.  By definition of the extended weight function, $\{W_{\rm ext} (x,y;\widehat\phi) -1 \}r = r\exp\{{\rm OR}(y \mid x;\widehat\gamma)\}V(x)$ with $V(x)={\rm pr}_{\rm ext}(r=0 \mid y=0,x;\widehat\phi)/{\rm pr}_{\rm ext}(r=1 \mid y=0,x;\widehat\phi)$. From (ii) of Lemma \ref{lem:1},  $\widehat E  [\{W_{\rm ext}(x,y;\widehat\phi )  -1\}r\{y-M_0(x;\widehat \beta,\widehat \gamma)\}]$ converges to $0$.  Thus,  $\widehat E  [\{W_{\rm ext}(x,y;\widehat\phi )r  -1\}\{y-M_0(x;\widehat \beta,\widehat \gamma)\}]$ converges to $0$, and
\begin{eqnarray*} 
\widehat\mu_2
	& = & 1/\widehat E \{W_{\rm ext}(x,y;\widehat\phi)r\} \cdot \widehat E  [\{W_{\rm ext}(x,y;\widehat\phi ) r -1\}\{y-M_0(x;\widehat \beta,\widehat \gamma)\}] \\
	&   & +1/\widehat E \{W_{\rm ext}(x,y;\widehat\phi)r\} \cdot \widehat E  [\{W_{\rm ext}( x,y;\widehat\phi) r -1\} \{ 
		M_0(x;\widehat\beta,\widehat\gamma) - \widehat\mu_{\rm reg}\} ] \\
	&   & + 1/\widehat E \{W_{\rm ext}(x,y;\widehat\phi)r\} \cdot \widehat E (y-\widehat\mu_{\rm reg}) + \widehat\mu_{\rm reg}
%	\text{Z linear in Y is not a special case}\\
\end{eqnarray*} 
converges to the true outcome mean in probability.

\item Double robustness of $\widehat\mu_3$.
If ${\rm pr}(r=1 \mid x,y=0;\alpha)$ is correct, from (i) of Lemma \ref{lem:1}, $\widehat E  [\{W(x,y;\widehat\alpha,\widehat\gamma)r-1\}\{y-M_{\rm 0ext}(x;\widehat\psi)\}]$ converges to $0$. Note equation \eqref{eq:rext},  we have that $\widehat E  [(1-r)\{y-M_{\rm 0ext}(x;\widehat\psi)\}]$ converges to $0$. Thus, $\widehat\mu_3=\widehat E     \{(1-r)M_{\rm 0ext}(x;\widehat\psi)  + ry\}$ converges to the true outcome mean.
 If $f(z,y \mid r=1,x;\beta)$ is correct, then  $\widehat E [(1-r)\{y-M_0(x;\widehat\beta,\widehat\gamma)\}]$ converges to $0$.  Since
$\{W (x,y;\widehat\alpha,\widehat\gamma) -1 \}r = r\exp\{{\rm OR}(y \mid x;\widehat\gamma)\}V(x)$ with $V(x)={\rm pr}(r=0 \mid y=0,x;\widehat\alpha)/{\rm pr}(r=1|y=0,x;\widehat\alpha)$,
from (ii) of Lemma \ref{lem:1},
$\widehat E  [\{W(x,y;\widehat\alpha,\widehat\gamma) - 1\}r\{ y - M_{\rm 0ext}(x;\psi=0)\} ] = \widehat E  [\{W( x,y;\widehat\alpha,\widehat\gamma) - 1\}r\{ y - M_0(x;\widehat\beta,\widehat\gamma)\} ]$  converges to $0$.  That is,  $\psi=0$ is a solution of the probability limit of equation \eqref{eq:rext}. Thus, the solution of equation \eqref{eq:rext}, $\widehat\psi$ converges to $0$, and  $\lim_{n\rightarrow +\infty} \widehat E     \{(1-r)M_{\rm 0ext}(x;\widehat\psi)  + ry\} =\lim_{n\rightarrow +\infty} \widehat E     \{(1-r)M_0(x;\widehat\beta,\widehat\gamma)  + ry\} =  E(Y)$.

\end{enumerate}
\end{proof}

\subsection*{Regularity conditions for   model (1)}
The full data law is identifiable if either $f(y|z,x,r=1)$ or $f(y|z,x,r=0)$ follows the location-scale  model (1), and the corresponding density function $f_{r=1}$ or $f_{r=0}$ satisfies the following  conditions:
\begin{enumerate}
\item[(a)] the characteristic function $\varphi(t)$ of the density function $f(v)$ satisfies $0<|\varphi(t)|<C\exp(-\delta | t|)$ for $t\in \mathbb{R}$ and some constants $C,
\delta>0$;

\item[(b)] conditional on $x$, $\mu(z,x)$, $\sigma(z,x)$ are continuously differentiable and integrable with respect to $z$;  $f(v)$   is continuously differentiable, and $\int_{-\infty}^{+\infty}|v \cdot \partial f(v)/\partial v|^2 dv$ is finite;

\item[(c)]  there exist some linear   one-to-one mapping $M: f\{(v-a)/b\}\longmapsto h(t,a,b)$ and  some value $-\infty\leq t_0\leq +\infty$ such that
$\lim_{t\rightarrow t_0}h(t,a,b)/h(t,a',b')$ either equals zero or infinity  for any $a,a'\in \mathbb R$, $b,b'>0$ with $(a,b)\neq(a',b')$. 
\end{enumerate}
Many commonly-used models satisfy  conditions (a)-(c), for example,    the Gaussian models with $f$  the standard normal density function,
$M$  the inverse Laplace transform,   $h(t,a,b)$  the moment-generating function of a normal density function with  mean $a$ and variance $b^2$, and $t_0=+\infty$.

\bibliographystyle{biometrika}
\bibliography{CausalMissing}

\begin{thebibliography}{19}
\expandafter\ifx\csname natexlab\endcsname\relax\def\natexlab#1{#1}\fi

\bibitem[{Bang \& Robins(2005)}]{bang2005doubly}
\textsc{Bang, H.} \& \textsc{Robins, J.~M.} (2005).
\newblock Doubly robust estimation in missing data and causal inference models.
\newblock \textit{Biometrics} \textbf{61}, 962--973.

\bibitem[{D'Haultfoeuille(2010)}]{d2010new}
\textsc{D'Haultfoeuille, X.} (2010).
\newblock A new instrumental method for dealing with endogenous selection.
\newblock \textit{Journal of Econometrics} \textbf{154}, 1--15.

\bibitem[{Higgins et~al.(2008)Higgins, White \& Wood}]{higgins2008imputation}
\textsc{Higgins, J.~P.}, \textsc{White, I.~R.} \& \textsc{Wood, A.~M.} (2008).
\newblock Imputation methods for missing outcome data in meta-analysis of
  clinical trials.
\newblock \textit{Clinical Trials} \textbf{5}, 225--239.

\bibitem[{Ibrahim et~al.(2001)Ibrahim, Lipsitz \& Horton}]{ibrahim2001using}
\textsc{Ibrahim, J.~G.}, \textsc{Lipsitz, S.~R.} \& \textsc{Horton, N.} (2001).
\newblock Using auxiliary data for parameter estimation with non-ignorably
  missing outcomes.
\newblock \textit{Journal of the Royal Statistical Society: Series C (Applied
  Statistics)} \textbf{50}, 361--373.

\bibitem[{Kang \& Schafer(2007)}]{kang2007demystifying}
\textsc{Kang, J.~D.} \& \textsc{Schafer, J.~L.} (2007).
\newblock Demystifying double robustness: A comparison of alternative
  strategies for estimating a population mean from incomplete data.
\newblock \textit{Statistical Science} \textbf{22}, 523--539.

\bibitem[{Kott(2014)}]{kott2014}
\textsc{Kott, P.} (2014).
\newblock Calibration weighting when model and calibration variables can
  differ.
\newblock In \textit{Contributions to Sampling Statistics}, Contributions to
  Statistics. Springer International Publishing, pp. 1--18.

\bibitem[{Miao et~al.(2015)Miao, Ding \& Geng}]{miao2014normal}
\textsc{Miao, W.}, \textsc{Ding, P.} \& \textsc{Geng, Z.} (2015).
\newblock Identifiability of normal and normal mixture models with nonignorable
  missing data.
\newblock \textit{Journal of the American Statistical Association}
  \textbf{accepted}.

\bibitem[{{Miao} et~al.(2015){Miao}, {Tchetgen Tchetgen} \&
  {Geng}}]{2015arXiv150902556M}
\textsc{{Miao}, W.}, \textsc{{Tchetgen Tchetgen}, E.} \& \textsc{{Geng}, Z.}
  (2015).
\newblock {Identification and doubly robust estimation of data missing not at
  random with a shadow variable}.
\newblock \textit{ArXiv:1509.02556} .

\bibitem[{Robins et~al.(2008)Robins, Li, Tchetgen~Tchetgen, van~der Vaart
  et~al.}]{robins2008higher}
\textsc{Robins, J.}, \textsc{Li, L.}, \textsc{Tchetgen~Tchetgen, E.},
  \textsc{van~der Vaart, A.} et~al. (2008).
\newblock Higher order influence functions and minimax estimation of nonlinear
  functionals.
\newblock In \textit{Probability and Statistics: Essays in Honor of David A.
  Freedman}, vol.~2. Institute of Mathematical Statistics, pp. 335--421.

\bibitem[{Robins et~al.(2007)Robins, Sued, Lei-Gomez \&
  Rotnitzky}]{robins2007comment}
\textsc{Robins, J.}, \textsc{Sued, M.}, \textsc{Lei-Gomez, Q.} \&
  \textsc{Rotnitzky, A.} (2007).
\newblock Comment: Performance of double-robust estimators when ``inverse
  probability" weights are highly variable.
\newblock \textit{Statistical Science} \textbf{22}, 544--559.

\bibitem[{Robins et~al.(1994)Robins, Rotnitzky \& Zhao}]{robins1994estimation}
\textsc{Robins, J.~M.}, \textsc{Rotnitzky, A.} \& \textsc{Zhao, L.~P.} (1994).
\newblock Estimation of regression coefficients when some regressors are not
  always observed.
\newblock \textit{Journal of the American Statistical Association} \textbf{89},
  846--866.

\bibitem[{Scharfstein et~al.(1999)Scharfstein, Rotnitzky \&
  Robins}]{scharfstein1999adjusting}
\textsc{Scharfstein, D.~O.}, \textsc{Rotnitzky, A.} \& \textsc{Robins, J.~M.}
  (1999).
\newblock Adjusting for nonignorable drop-out using semiparametric nonresponse
  models.
\newblock \textit{Journal of the American Statistical Association} \textbf{94},
  1096--1120.

\bibitem[{Tsiatis(2006)}]{tsiatis2007semiparametric}
\textsc{Tsiatis, A.} (2006).
\newblock \textit{Semiparametric Theory and Missing Data}.
\newblock New York: Springer.

\bibitem[{Van~der Laan \& Robins(2003)}]{van2003unified}
\textsc{Van~der Laan, M.~J.} \& \textsc{Robins, J.~M.} (2003).
\newblock \textit{Unified Methods for Censored Longitudinal Data and
  Causality}.
\newblock New York: Springer.

\bibitem[{Vansteelandt et~al.(2007)Vansteelandt, Rotnitzky \&
  Robins}]{vansteelandt2007estimation}
\textsc{Vansteelandt, S.}, \textsc{Rotnitzky, A.} \& \textsc{Robins, J.}
  (2007).
\newblock Estimation of regression models for the mean of repeated outcomes
  under nonignorable nonmonotone nonresponse.
\newblock \textit{Biometrika} \textbf{94}, 841--860.

\bibitem[{Vermeulen \& Vansteelandt(2014)}]{vermeulen2014biased}
\textsc{Vermeulen, K.} \& \textsc{Vansteelandt, S.} (2014).
\newblock Biased-reduced doubly robust estimation.
\newblock \textit{Journal of the American Statistical Association}
  \textbf{accepted}.

\bibitem[{Wang et~al.(2014)Wang, Shao \& Kim}]{wang2014instrumental}
\textsc{Wang, S.}, \textsc{Shao, J.} \& \textsc{Kim, J.~K.} (2014).
\newblock An instrumental variable approach for identification and estimation
  with nonignorable nonresponse.
\newblock \textit{Statistica Sinica} \textbf{24}, 1097--1116.

\bibitem[{Zahner et~al.(1992)Zahner, Pawelkiewicz, DeFrancesco \&
  Adnopoz}]{zahner1992children}
\textsc{Zahner, G.~E.}, \textsc{Pawelkiewicz, W.}, \textsc{DeFrancesco, J.~J.}
  \& \textsc{Adnopoz, J.} (1992).
\newblock Children's mental health service needs and utilization patterns in an
  urban community: an epidemiological assessment.
\newblock \textit{Journal of the American Academy of Child \& Adolescent
  Psychiatry} \textbf{31}, 951--960.

\bibitem[{Zhao \& Shao(2014)}]{zhao2014semiparametric}
\textsc{Zhao, J.} \& \textsc{Shao, J.} (2014).
\newblock Semiparametric pseudo likelihoods in generalized linear models with
  nonignorable missing data.
\newblock \textit{Journal of the American Statistical Association}
  \textbf{accepted}.

\end{thebibliography}

\end{document}